\documentclass[10pt]{amsart}
\usepackage{amscd,amssymb}
\usepackage[all]{xy}
\begin{document}

\newtheorem{prop}{Proposition}
\newtheorem{lemma}[prop]{Lemma}
\newtheorem{thm}[prop]{Theorem}
\newtheorem{defn}[prop]{Definition}
\newtheorem{cor}[prop]{Corollary} 
\newtheorem{rmk}[prop]{\bf Remark}
\newtheorem{rmks}[prop]{\bf Remarks}
\newtheorem{ex}[prop]{Example}
\newtheorem{exs}[prop]{Examples}

\newcommand{\tip}{\operatorname{tip}\nolimits}
\newcommand{\grb}{Gr\"obner }
\newcommand{\extto}{\xrightarrow}
\newcommand{\cA}{{\mathcal A}}
\newcommand{\cB}{{\mathcal B}}
\newcommand{\cD}{{\mathcal D}}
\newcommand{\cY}{{\mathcal Y}}
\newcommand{\cU}{{\mathcal U}}
\newcommand{\cJ}{{\mathcal J}}
\newcommand{\cI}{{\mathcal I}}
\newcommand{\cF}{{\mathcal F}}
\newcommand{\cX}{{\mathcal X}}
\newcommand{\cP}{{\mathcal P}}
\newcommand{\cQ}{{\mathcal Q}}
\newcommand{\cM}{{\mathcal M}}
\newcommand{\cK}{{\mathcal K}}
\newcommand{\cH}{{\mathcal H}}
\newcommand{\cT}{{\mathcal T}}
\newcommand{\cG}{{\mathcal G}}
\newcommand{\cE}{{\mathcal E}}
\newcommand{\cC}{{\mathcal C}}
\newcommand{\cO}{{\mathcal O}}
\newcommand{\dpp}{\prime\prime}
\newcommand{\br}{{\bf r}}
\newcommand{\bx}{$_{\fbox{}}$\vskip .2in}
\newcommand{\Gr}{\operatorname{Gr}\nolimits}
\newcommand{\gr}{\operatorname{gr}\nolimits}
\newcommand{\Mod}{\operatorname{\mathbf{Mod}}\nolimits}
\newcommand{\smod}{\operatorname{\mathbf{mod}}\nolimits}
\newcommand{\End}{\mbox{End}}
\newcommand{\Hom}{\mbox{Hom}}
\renewcommand{\Im}{\mbox{Im}}
\newcommand{\Ext}{\operatorname{Ext}\nolimits}
\newcommand{\pd}{\mbox{pd}}
\newcommand{\Tor}{\mbox{Tor}}
\renewcommand{\dim}{\mbox{dim}}
\newcommand{\gldim}{\mbox{gl.dim}}
\newcommand{\op}{^{\mbox{op}}}
\newcommand{\pr}{^{\prime}}
\newcommand{\f}{\operatorname{fin}}
\newcommand{\Sy}{\operatorname{syz}}
\newcommand{\semi}{\mathbin{\vcenter{\hbox{$\scriptscriptstyle|$}}
\;\!\!\!\times }}
\newcommand{\fralg}{K\!\!<\!\!x_1,\dots,x_n\!\!>}
\newcommand{\Efin}{\operatorname{Efin}\nolimits}
\newcommand{\Lin}{\operatorname{Lin}\nolimits}
\newcommand{\Tlin}{\operatorname{Tlin}\nolimits}
\newcommand{\WKS}{\operatorname{WKS}\nolimits}
\newcommand{\cx}{\operatorname{cx}\nolimits}
\newcommand{\pcx}{\operatorname{pcx}\nolimits}
\newcommand{\tpcx}{\operatorname{tpcx}\nolimits}
\newcommand{\mpcx}{\operatorname{mpcx}\nolimits}
\newcommand{\tmpcx}{\operatorname{tmpcx}\nolimits}
\newcommand{\mto}{\hookrightarrow}
\newcommand{\AR} {Auslander-Reiten }

\newcommand{\rank}{\operatorname{rank}\nolimits}
\newcommand{\e}{^{\text{e}}}
\newcommand{\od}{^{\text{o}}}
\newcommand{\dd}{(\delta\e,\delta\od)}
\newcommand{\refl}[1]{{\operatorname{\mathbf {Refl}}\nolimits}({#1},{\dd})}
\newcommand{\hrefl}[1]{{\operatorname{\mathbf {Refl}}\nolimits}({#1},\rho)}

\newcommand{\href}[2]{{\operatorname{\mathbf {Refl}}\nolimits}({#1},{#2})}
\newcommand{\shref}[2]{{\operatorname{\mathbf {refl}}\nolimits}({#1},{#2})}
\newcommand{\rep}[1]{\operatorname{\mathbf{Rep}}\nolimits(K,{#1})}
\newcommand{\srep}[1]{\operatorname{\mathbf{rep}}\nolimits(K,{#1})}

\newcommand{\bGamma}{\overline{\Gamma}}
\newcommand{\bGammap}{\overline{\Gamma'}}
\newcommand{\cring}{${\mathbf C}$-{\bf Ring}}
\newcommand{\AP}{\operatorname{AP}\nolimits}

\title
{$d$-Koszul algebras, $2$-$d$ determined algebras and $2$-$d$-Koszul algebras}
\author{Edward L. Green}\address{Department of Mathematics\\
Virginia Tech University\\ 
Blacksburg, VA 24061\\ USA}\email{green@math.vt.edu}
\author{E.\ N.\ Marcos}
\address{Departamento de Matem\'atica, Universidade de S\~ao
Paulo, IME-USP\\ Caixa Postal 66.281\\ S\~ao Paulo -- SP, 05315--970, Brasil}\email{\tt
enmarcos@ime.usp.br}
\thanks{The first author is partially supported by a grant from NSA, and 
the second author has a productivity grant from CNPq-Brasil and a Projeto Tem\'atico from FAPESP} 


\begin{abstract} The relationship between an algebra and its associated monomial
algebra is investigated when at least one of the algebras is $d$-Koszul.  It is shown
that an algebra which has a reduced \grb basis that is composed of homogeneous  elements
of degree $d$ is $d$-Koszul
if and only if its associated monomial algebra is $d$-Koszul.  The class
of $2$-$d$-determined algebras and the class $2$-$d$-Koszul algebras are introduced.
In particular, it shown that $2$-$d$-determined monomial algebras are $2$-$d$-Koszul
algebras and the structure of the ideal of relations of such an algebra is completely
determined. 
\end{abstract}

\maketitle


\section{Introduction}\label{intro} 
This paper focuses on the study of classes of graded algebras such that
graded projective resolutions of the semisimple
part of such graded algebras have special properties. We also investigate the
relationship between the algebra having projective resolutions with certain special 
properties and the structure of the Ext-algebra of the semisimple part of such an
algebra.  In the past,
some of the strongest results have been obtained for Koszul algebras, a special class
of graded algebras that have occurred 
in many
diverse settings.  Generalizations of Koszul algebras, for example, $d$-Koszul
algebras \cite{b,gmmz}, have recently been studied.  
In this paper, we continue the investigation 
of $d$-Koszul algebras and begin a study of
a new classes of algebras which we call $2$-$d$-determined algebras and
$2$-$d$-Koszul algebras.   We begin by 
summarizing the major results of the paper.  Precise definitions for many of the
terms used can be found later in this section and the next section.

In the summary below, we let $\Lambda=K\Gamma/I$, where $K$ is a field, $\Gamma$ a finite
quiver, $K\Gamma$ the path algebra, and $I$ an ideal generated by length homogeneous elements.
Let $J$ be the ideal in $K\Gamma$ generated by the arrows of $\Gamma$ and assume that
$I\subset J^2$.  The length grading of $K\Gamma$ induces a positive $\mathbb Z$-grading of
$\Lambda=\Lambda_0\oplus \Lambda_1\oplus\cdots$, where $\Lambda_0$ is the $K$-space
spanned by the vertices of $\Gamma$.  In particular, $\Lambda_0\cong\Lambda/(J/I)$ and
$J/I\cong \Lambda_1\oplus\Lambda_2\oplus\cdots$ is graded Jacobson radical of $\Lambda$.

After the summary of results, this section ends with the introduction of notation, background, and
a brief overview of the theory of
\grb bases for path algebras.
In Section \ref{ags sec} we recall constructions of projective resolutions
found in \cite{ag, gs}, which we call the `AGS resolution', and also review
the general approach to the structure of projective resolutions
found in \cite{gz}.  Given $\Lambda=K\Gamma/I$, using the theory of \grb bases, we associate 
a monomial
algebra, $\Lambda_{mon}$,
to $\Lambda$, where by `monomial algebra', we mean a quotient of a path
algebra by an ideal that can be generated by a set of path. In this case, 
$\Lambda_{mon}=K\Gamma/I_{mon}$,
where $I_{mon}$ is the ideal generated by the `tips' or `leading terms' of $I$. 
One of the main objectives of the paper is the study of the interrelationship of
$\Lambda$ and $\Lambda_{mon}$.
In Section \ref{d koz}, we turn our attention to 
$d$-Koszul algebras, which were introduced by Berger \cite{b}.
Let $\mathbb N$ denote the natural numbers $\{0,1,2,\dots\}$ and let $d\in \mathbb N$
with $d\ge 2$.
Consider the function $\delta\colon\mathbb N\to \mathbb N$ 
defined by
 \[\delta(n)=\lbrace \begin{array}{l}
 \frac{n}{2}d\mbox{ if }n\mbox{ is even}\\ \\ \frac{n-1}{2}d +1\mbox{ if }n\mbox{ is odd}
 \end{array}. \]
We say that $\Lambda =K\Gamma/I$ is a {\it $d$-koszul algebra} if the $n^{th}$-projective module
in a minimal graded projective $\Lambda$-resolution of $\Lambda_0$ can be generated in degree $\delta(n).$
More generally, if $F\colon \mathbb N\to \mathbb N$, we say that $\Lambda$ is
{\it $F$-determined}, (respectively {\it weakly $F$-determined}) 
in case 
the $n^{th}$-projective module
in a minimal graded projective $\Lambda$-resolution of $\Lambda_0$ 
can be generated in degree $F(n)$, (resp. $\le F(n)$), for
all $n\in\mathbb N$.  The notion of $F$-determined algebras was introduced in \cite{gm} and
also investigated in \cite{gs2}.
Proposition \ref{f-det} shows that, in particular, if $\Lambda_{mon}$ is weakly $F$-determined,
then so is $\Lambda$.
We use this result to show that if $\Lambda_{mon}$ is a $d$-Koszul
algebra then so is $\Lambda$ in Corollary \ref{d-kosz-cor}. Theorem \ref{monomial} 
gives a partial converse, showing that if $\Lambda=K\Gamma/I$ and $I$ has a reduced
\grb basis concentrated in degree $d$, then $\Lambda_{mon}$ is a $d$-Koszul algebra.
Theorem \ref{summ} summarizes the main results of the section.

In Section \ref{2-d det}, we introduce the class of $2$-$d$-determined algebras.
We say that $\Lambda$
is {\it $2$-$d$-determined} if $I$ can be generated by homogeneous elements of
degrees $2$ and $d$, and $\Lambda$ is weakly $\delta$-determined.  We say a
$2$-$d$-determined algebra is {\it $2$-$d$-Koszul} if the Ext-algebra,
$\bigoplus_{n\ge 0}\Ext^n_{\Lambda}(\Lambda_0,\Lambda_0)$, can be
finitely generated.  In this section, we mainly consider the case where
$\Lambda$ is a monomial algebra.  Theorem \ref{mono-2d} proves that if
$\Lambda=K\Gamma/I$ and $I$ is generated by paths of lengths $2$ and $d$,
then $\Lambda$ is a $2$-$d$-determined algebra if and only if $K\Gamma/\langle \cG_d\rangle$
is a $d$-Koszul algebra, where $\cG_d$ denotes
the set of paths of length $d$ in a minimal generating set of $I$. 
In Theorem \ref{2d-mono-koz}, we show that a monomial algebra with generators
in degrees $2$ and $d$ is $2$-$d$-determined if and only if the Ext-algebra,
$\bigoplus_{n\ge 0}\Ext^n_{\Lambda}(\Lambda_0,\Lambda_0)$,
can be generated in degrees $0$, $1$, and $2$. Algebras, whose Ext-algebra can
be generated in degrees $0$, $1$, and $2$ have been called {\it K2 algebras} 
by Cassidy and Shelton \cite{cs}.

In the final section, Section \ref{2d-grobner}, we study $2$-$d$-determined algebras
in general.  Proposition \ref{gen-2d} shows that if $\Lambda_{mon}$ is $2$-$d$-Koszul
then $\Lambda$ is $2$-$d$-determined.  The next result is the main result of the section.

\vskip .2in
\noindent {\bf Theorem 18} {\em
Let $\Lambda = K\Gamma/I$, where $I$ is a homogeneous 
ideal in $K\Gamma$, and let $>$ be an admissible order on $\cB$.  Suppose that the reduced
Gr\"obner basis $\cG$  of $I$ with respect to $>$ satisfies $\cG =\cG_2 \cup \cG_d$
where $\cG_2$ consists of homogeneous elements of degree $2$ and
$\cG_d$ consists of homogeneous elements of degree $d$, where $d\ge 3$.
Then $\Lambda$ is $2$-$d$-determined
if $K\Gamma/\langle\tip( \cG_d)\rangle$ is a $d$-Koszul algebra.   
}
\vskip .2in
\noindent
Section \ref{2d-grobner} ends with some open questions.

We end this section with some definitions and notations that will be used throughout
the remainder of the paper.
We always let $\Gamma$ denote a finite quiver and $K\Gamma$ its path algebra
 over a fixed field $K$. The $K$-algebra $K\Gamma$ is naturally a positively
 $\mathbb Z$-graded algebra, where, if $n$ is a nonnegative integer, then $(K\Gamma)_n$ 
 denotes
 the homogeneous component of $K\Gamma$ which is the vector space with basis the set of
 paths of length $n$.  We denote the length of a path $p$ by $\ell(p)$ and
 let $\Gamma_n$ denote
 the set of directed paths of length $n$ in $\Gamma$; in particular, $\Gamma_0$ is 
 the set of vertices of $\Gamma$ and $\Gamma_1$ is the set of arrows in $\Gamma$.  
 We call this the {\it length grading}
 of $K\Gamma$ and say an element of $K\Gamma$ is {\it homogeneous} if all the paths 
 occurring in the element have the same length. 
 In particular,  if $f\in (K\Gamma)_n$, then we say that $f$ is {\it homogeneous of (length)
 degree $n$} and write $\ell(f)=n$.
 If $I$ is an ideal in $K\Gamma$, we say 
 that $I$ is a {\it homogeneous ideal} if $I$ can be generated by homogeneous elements.
 Clearly, if $I$ is a homogeneous ideal in $K\Gamma$, then $K\Gamma/I$ has an grading
 induced from the length grading of $K\Gamma$, and we call this the {\it length grading
 on $K\Gamma/I$ induced by the length grading on $K\Gamma$}, or simply, the
 {\it induced length grading on $K\Gamma/I$}.
 
 If an ideal $I$ can be generated by a set of paths in $\Gamma$, then we say that
 $I$ is a {\it monomial ideal} and that $\Lambda=K\Gamma/I$ is a {\it monomial
 algebra}.  Since every monomial ideal is a homogeneous ideal, every monomial
 algebra has an induced length grading.
 
As mentioned earlier, we  denote by $J$,  the ideal of $K\Gamma$ generated by the
 arrows of $\Gamma$.  By `module', we  mean `left module' unless
 otherwise stated.
 If  $\Lambda$ = $K\Gamma/I$, where $I$ is an ideal contained
 in $J$ and we denote by $\Lambda_0$, the semisimple $\Lambda$-module $\Lambda_0$ =
 $({\Lambda/I})/ (J/I)$.  Suppose further that
$I$ is a homogeneous ideal and $\Lambda=K\Gamma/I$ is
given the induced length grading. Note that $J/I$ is the graded Jacobson
radical of $\Lambda$.  The $\Lambda$-module $\Lambda_0$ 
will also be viewed as a graded $\Lambda$-module
whose support is concentrated in degree $0$.
If $S_1,\dots, S_n$ is a full set of nonisomorphic simple
$\Lambda$-modules, then $\Lambda_0\cong \oplus_{i=1}^nS_i$, as an (ungraded) $\Lambda$-module.
We also note that, in the category
of graded $\Lambda$-modules, $\Lambda_0$ has a {\it minimal graded projective resolution}
\[\cdots \to P^2\to P^1\to P^0\to \Lambda_0\to 0,\]
in the sense that each $P^n$ is a graded projective $\Lambda$-module, each
map $P^n\to P^{n-1}$ is a degree $0$ homomorphism, and, for each
$n\ge 1$, the image of $P^n$ in $P^{n-1}$ is contained in $(J/I)P^{n-1}$.

If $\Lambda=K\Gamma/I$, for some homogeneous ideal $I$ in $K\Gamma$, and $v\in \Gamma_0$, then
we view $\Lambda v$ as an indecomposable graded $\Lambda$-module generated in degree
$0$ by $v$.  If $M=\oplus_{i\in\mathbb Z}M_i$ is a graded $\Lambda$-module,
then we let $M[n]$ denote the $n^{th}$-shift of $M$; that is, 
$M[n]=\oplus_{i\in\mathbb Z}N_i$, where $N_i=M_{i+n}$.  It is well known, for example,
see \cite{g},
that every graded indecomposable
projective $\Lambda$-modules is isomorphic to $\Lambda v[n]$, for some unique $n\in \mathbb
Z$
and $v\in \Gamma_0$ and that every finitely generated
graded projective $\Lambda$-module
can be written as direct sum of projective modules of the form $\Lambda v[n]$.

 Given
 a set $X$ in $K\Gamma$, we denote by $\langle X\rangle$, the two sided ideal 
in $K\Gamma$ generated by $X$. 
We will freely use the terminology and results about \grb bases for
path algebras found in \cite{g2}.
For the reader's benefit, we recall some of the definitions.   We say a nonzero element $x\in K\Gamma$
is {\it uniform} if there exist vertices $v,w\in \Gamma_0$ such that
$vx=x=xw$.  Note that any nonzero element of $K\Gamma$ is a sum of uniform elements and
that any ideal in $K\Gamma$ can be generated by uniform elements.

Let $\cB=\cup_{n\ge 0}\Gamma_n$ be the
set of paths in $\Gamma$. We say a well ordering $>$ on $\cB$ is an
{\it admissible order} if the following conditions hold for all $p,q,r,s\in \cB$.
\begin{enumerate}
\item If $p>q$, then $rp>rq$, if both are nonzero.
\item If $p>q$, then $pr>qr$, if both are nonzero.
\item If $p=qrs$, then $p\ge r$.
\end{enumerate}
If $x\in K\Gamma$, then $x=\sum_{p\in\cB}\alpha_pp$, where $\alpha_p\in K$ and almost
every $\alpha_p=0$.  If $x\ne 0$, the {\it tip of $x$}, denoted $\tip(x)$, is the
path $p\in\cB$ such that $\alpha_p\ne 0$ and  $p\ge q$, for all $q$ such that $\alpha_q\ne 0$.
If $X\subset K\Gamma$, then $\tip(X)=\{\tip(x)\mid x\in X\setminus\{0\}\}$.  We say a 
path $p$ {\it occurs in $\sum_{q\in\cB}\alpha_qq\in K\Gamma$} if $\alpha_p\ne 0$.

Fix an admissible order $>$ on $\cB$ and let $I$ be an ideal in $K\Gamma$.  We say a
set $\cG$ of nonzero uniform elements in $I$ is a {\it \grb basis of $I$ (with respect to $>$)}
if $\langle \tip(\cG)\rangle=\langle \tip(I)\rangle$.  We say a \grb basis $\cG$ of $I$ 
is the {\it reduced \grb basis} if, for every $g\in \cG$, the coefficient of $\tip(g)$ is
$1$ and, if $p$ is a path occurring in 
$g$ and $p$ contains
a subpath $t$, where $t=\tip(g')$,
for some $g'\in\cG$, then $g=g'$.
We note that given $>$ and an ideal $I$ in $K\Gamma$, the reduced \grb basis of $I$ exists
and is unique.
Using the Buchberger algorithm as 
generalized for path algebras, one can see  
that if $I$ is a homogeneous ideal in $K\Gamma$, the reduced \grb
basis of $I$  consists of homogeneous uniform elements. 

Now suppose that $\cG$ is the reduced \grb basis of an ideal $I$ with respect
to $>$.  Let $I_{mon}$, the {\it associated monomial ideal to $I$}, be the ideal in $K\Gamma$ generated by the tips
of $\cG$.  If $\Lambda=K\Gamma/I$, we let $\Lambda_{mon}=K\Gamma/I_{mon}$. 
We call $\Lambda_{mon}$ the {\it associated monomial algebra of $\Lambda$} with
respect to $>$.  We note that if $I$ is a monomial ideal, then $I=I_{mon}$, and
this is independent of the choice of $>$; whereas, if $I$ is not a monomial ideal, then
$I_{mon}$ usually depends on the choice of admissible order.

\section{The AGS resolution}\label{ags sec}

Although the proofs of the results in this section appear in other papers, they are not
stated or combined together in fashion we need throughout the remainder of the
paper.  Hence we have included this survey for the readers benefit.

In both \cite{ag} and \cite{gs}, methods for constructing a
projective $\Lambda$-resolution of $\Lambda_0$ are given and we will call such
a constructed resolution the {\it AGS resolution}. The reader may check these constructed 
resolutions of $\Lambda_0$
are, in fact, the same; although in \cite{gs}, resolutions of a larger class of modules,
that includes $\Lambda_0$, are given.
Both methods employ an admissible order on $\cB$ and a \grb basis of $I$ (with respect to the
chosen admissible
order).   Furthermore, the reader may check
that, if $I$ can be generated by length homogeneous elements, 
then the AGS resolution is, in fact, a resolution
in the category of graded $\Lambda$-modules, see \cite{gsz}. 
In general, the AGS resolution is not
minimal, but, if the \grb basis is finite, the projective modules occurring in
the resolution, viewed as graded $\Lambda$-modules, 
can be written as finite direct sums of projective $\Lambda$-modules
of the form $\Lambda v[n]$, where $v$ is a vertex in $\Gamma$, and $n\in \mathbb Z$.  
Note that,  if $I$ has a \grb basis that consists only of paths (of
length at least 2), then the AGS resolution is minimal.   

In a path algebra $K\Gamma$, if $x\in K\Gamma$ is a nonzero element such that
$vx=x$, where $v\in \Gamma_0$, then we let $o(x)=v$.  Similarly, if $xv=x$, where
$v\in \Gamma_0$, then we let $t(x)=v$.

Let $>$ be an admissible order, $I$ a homogeneous ideal in $K\Gamma$, $\Lambda=K\Gamma/I$,
and $\cG$ be the reduced \grb basis for $I$. Suppose that  $\cG=\{g^2_i\}_{i\in \cI}$, for some
index set $\cI$.  If $v$ is a vertex, by abuse of notation, we will let
$\Lambda v$ also denote the graded projective $\Lambda$-module
generated by $v$ with $v$ in degree $0$.
Suppose that 
\[ \cdots \to Q^2\to Q^1\to Q^0\to \Lambda_0\to 0\]
is the AGS (graded) resolution of $\Lambda_0$.  Then
$Q^0=\oplus_{v\in \Gamma_0}\Lambda v, Q^1=\oplus_{a\in \Gamma_1}\Lambda o(a)[-1]$, and
$Q^2=\oplus_{i\in \cI}\Lambda o(g^2_i)[-\ell(g^2_i)]$.

We briefly describe the structure of
$Q^3$, leaving details to be found in \cite{gs}. For this, we need a few more
definitions.  If $p,q\in \cB$, we say
$p$ {\it overlaps} $q$ if there are paths $r$ and $s$ such that $pr=sq$ and $\ell(s)<\ell(p)$, and that
the overlap is {\it proper} if $\ell(r)\ge 1$ and $\ell(s)\ge 1$.  We say
$q$ is a {\it subpath} of $p$ if $p=rqs$ for some paths $r$ and $s$, and that
$q$ is a {\it proper subpath} of $p$, if $p=rqs$, for some paths $r$ and $s$, with
$\ell(r)\ge 1$ and $\ell(s)\ge 1$.  
As remarked earlier,
since $\cG$ is a reduced \grb basis, if
$g^2_i=\sum_{j=1}^m\alpha_jq_j$, where
each $\alpha_j$ is a nonzero element of $K$ and the $q_j$'s are distinct paths,
then if $s\ne i$, $\tip(g^2_s)$ is not a subpath of $q_j$, for $j=1,\dots, m$.

If $\rho$ is a set of paths of length at least 2, and $t,q\in \cB$, then 
we say a path $p$ is {\it the maximal overlap
of $t$ with $q$ with respect to $\rho$} if the following two conditions hold.
\begin{enumerate}
\item If $t$ overlaps $q$ such that there exist paths $s,s'\in\cB$ with  
$\ell(s')\ge 1$ with $p=s'q=ts$.
\item 
For all $t'\in \rho$, $t'$ is not a proper subpath of $p$.
\end{enumerate}
In this case, we say that $p$ {\it maximally overlaps} $q$ with respect to $\rho$.

We will be interested in maximal overlaps of elements of $\tip(\cG)$ with
various paths with respect $\tip(\cG)$.  In particular,
let \vskip .2in
\noindent
$T^3=\{p\in \cB\mid p \mbox{ is the maximal overlap of } t' \mbox{ with } 
t \mbox{ with respect to }\tip(\cG), \mbox{ where }$ \linebreak
\phantom{xxxxxxxxxxxxxxxxxxxxxxx}$t,t'\in\tip(\cG)\}.$
\vskip .2in
Note that if $\cG$ is a finite set, then $T^3$ is also a finite set. 
In general, we have the following.

\begin{prop}\label{Q3}{\mbox{\cite{gs}}} Let $\Lambda=K\Gamma/I$ where $I$ is a homogeneous ideal
in $K\Gamma$ and suppose that $>$ is an admissible order on $\cB$.  Let
\[\cdots \to Q^2\to Q^1\to Q^0\to \Lambda_0\to 0\]
be the AGS (graded) resolution of $\Lambda_0$.  Then, as graded
$\Lambda$-modules, 
\[Q^3=\oplus_{t\in T^3}\Lambda o(t)[-\ell(t)],\]
where $T^3=\{p\in \cB\mid p \mbox{ is the maximal overlap of } t' \mbox{ and } 
t, \mbox{ for some }t,t'\in\tip(\cG)\}$.
\end{prop}

The AGS resolution is a special case of projective $\Lambda$-resolutions of
modules which are studied in \cite{gsz}.  We recall some definitions and results
from that paper, since the perspective and notation 
developed there will be
used in some of the proofs that follow. For ease of notation, we will sometimes
denote $K\Gamma$ by $R$.

Let $M$ be a  $K\Gamma$-module and $m$ be a nonzero element of $M$. We say that
$m$ is \emph{left uniform} if there exist $u$ in $\Gamma_0$ such that $m=um$.
In this case, we let $o(m)=u$.
Note that if $\Gamma$ has a single vertex
then every nonzero of $M$ is left uniform.

Suppose that $M$ is a finitely generated $\Lambda$-module.
Then, as shown in \cite{gsz}, there exist $t_n$ and $u_n$ in
$\{0,1,2,\dots\}\cup \infty$ with $u_0=0$, $\{f^n_i\}_{i\in
T_n=[1,\dots,t_n]}$, and $\{{f^n_i}'\}_{i\in U_n=[1,\dots,u_n]}$ such
that
\begin{enumerate}
\item[(i)] Each $f^0_i$ is a left uniform element of $R$, for all
$i\in T_0$.
\item[(ii)] Each $f_i^n$ is in $\oplus_{j\in
T_{n-1}}Rf^{n-1}_j$ and is a left uniform element, for all $i\in T_n$
  and all $n\ge 1$. 
\item[(iii)] Each ${f_i^n}'$ is in $\oplus_{j\in T_{n-1}}If^{n-1}_j$ and
is a left uniform element for all $i\in U_n$ and all $n\ge 1$.
\item[(iv)] For each $n\ge 2$, \[(\oplus_{i\in T_{n-1}}Rf^{n-1}_i)
\cap (\oplus_{i\in T_{n-2}}I{f^{n-2}_i}) =(\oplus_{i\in T_n}R f^n_i)
\oplus(\oplus_{i\in U_n}R{f^n_i}').\]
\end{enumerate}

An explicit description of the tip set of $T_n$ for the AGS resolution of $\Lambda_0$, is given 
in Proposition \ref{apn} below.
The next result explains how the sets $\{f^n_i\}_{i\in T_n}$ and
$\{{f^n_i}'\}_{i\in U_n}$ give rise to a
projective $\Lambda$-resolution of $M$. 
We have the following isomorphisms:
\[\oplus_{i=1}^m Rf_i/\oplus_{i=1}^m If_i\cong\oplus_{i=1}^m (Rf_i/If_i)
\cong\oplus_{i=1}^m \Lambda o(f_i).\]

\begin{thm}[\cite{gsz}]\label{thm:fnresol} 
Let $M$ be a finitely generated $\Lambda$-module and suppose
that, for $n \ge 0$, $t_n$ and $u_n$ are in $\{0,1,2,\dots\}\cup
\infty$,  $\{f^n_i\}_{i\in T_n=[1,\dots,t_n]}$, and 
$\{{f^n_i}'\}_{i\in U_n=[1,\dots,u_n]}$ are chosen satisfying
\emph{(i)-(iv)} above.  
Let 
\[L^n=\bigoplus_{i\in
T_n}\Lambda o(f^n_i).\]
Then there exist maps $e^n\colon L^n\to L^{n-1}$ and a surjection $L^0\to M$ such that 
\[\cdots \extto{e^{n+1}}L^n\extto{e^n}L^{n-1}\extto{e^{n-1}}\cdots \extto{e^1}
L^0\extto{} M\extto{} 0\]
is a projective $\Lambda$-resolution of $M$.
\end{thm}

Although we do not use explicit descriptions of the maps $e^n$ in this paper,
we note that such descriptions can be found in \cite{gsz}. 
The AGS resolution is obtained  by constructing  particular $f^n_i$'s which satisfy
{(i)-(iv)}.  By Proposition \ref{Q3}, we see that, for this choice
of the $f^n_i$'s, $\{\tip(f^3_i)\}_{i\in T_3}$
is precisely the set of maximal overlaps $T^3$ defined ealier.
From this observation, we have the following useful result.

\begin{prop}\label{overlp} Let $>$ be an admissible order, 
$I$ a homogeneous ideal in $K\Gamma$, $\Lambda=K\Gamma/I$,
and $\cG$ be the reduced \grb basis for $I$.
Suppose that 
\[\cdots \to P^2\to P^1\to P^0\to \Lambda_0\to 0\] is a minimal graded projective 
$\Lambda$-resolution of $\Lambda_0$.
Then $P^3$ is isomorphic to $\bigoplus_{t\in (T^*)^3}\Lambda o(t)[-t]$, for some
subset $(T^*)^3$ of the set of maximal overlaps of $\cG$ with respect to $\cG$.
\end{prop}

\begin{proof} 
Let $\cdots \to Q^2\to Q^1\to Q^0\to\Lambda_0\to 0$ be the (graded) AGS resolution
of $\Lambda_0$.  By Proposition \ref{Q3}, $Q^3=\bigoplus_{t\in T^3}\Lambda o(t)[-t]$, where
$T^3$ is the set of maximal overlaps of $\cG$ with respect to $\cG$.  The result now follows
from \cite[Theorem 2.4]{gsz} after by applying the proof of  \cite[Theorem 2.3]{gsz}.
\end{proof}

In a similar fashion, the following more general result is a consequence of the proof of Theorem 2.3 and
Theorem 2.4 in \cite{gsz}.

\begin{prop}\label{genl-overlp} Let $>$ be an admissible order, 
$I$ a homogeneous ideal in $K\Gamma$, and $\Lambda=K\Gamma/I$.
Suppose that 
\[\cdots \to P^2\to P^1\to P^0\to \Lambda_0\to 0\] is a minimal graded projective 
$\Lambda$-resolution of $\Lambda_0$ and that
\[\cdots \to Q^2\to Q^1\to Q^0\to\Lambda_0\to 0\] 
is the (graded) AGS resolution
of $\Lambda_0$.  If $f\in K\Gamma$ is a homogeneous element of degree $d$, we
set $\ell(f)=d$
If $Q^n\cong \bigoplus_{f^n_j\in T_n}\Lambda o(f^n_j)[-\ell(f^n_j)]$, then
$P^n$ is isomorphic to $\bigoplus_{f^n_j\in (T_n)^*}\Lambda o(f^n_j)[-\ell(f^n_j)]$, 
for some
subset $(T_n)^*$ of $T_n$.
\end{prop}

We introduce some notation that will be needed later in the paper. Let $\rho$ be a set of paths
of length at least $2$ such that no path in $\rho$ is a subpath of any other path in $\rho$.
We define the sets $\AP(n)$ of {\it admissible paths of order $n$, with respect to $\rho$}.
First let $\AP(0)=\Gamma_0$, $\AP(1)=\Gamma_1$, and $\AP(2)=\rho$. Next, we let
$\AP(3)$ be the set of all
maximal overlaps of elements of $\rho$ with elements of $\rho$, with respect to
$\rho$.  Assuming $\AP(n-2)$ and
$\AP(n-1)$ have been defined.  Then define 
$\AP(n)$ to be the set of paths $a_n$ in $\Gamma$ which satisfy the following
conditions.
\begin{enumerate}
\item[(A1)] There are paths $a_{n-1}\in \AP(n-1)$ and $r\in\cB$ such that $\ell(r)\ge 1$ and
$a_n=ra_{n-1}$.
\item[(A2)] If $a_{n-1}=sa_{n-2}$ with $a_{n-2}\in \AP(n-2)$, then 
$rs=a_2s'$ for some $s'\in\cB$, some $a_2\in \rho$, and we have that $\ell(r)<\ell(a_2)$.
\item[(A3)] $a_2s'$ does not contain any element of $\rho$ as a proper 
subpath.
\end{enumerate}
The reader may check that for $n\ge 2$, given $a_n\in \AP(n)$, then there exist
unique paths $r\in\cB$ and $a_{n-1}\in\AP(n-1)$ such that $a_n=ra_{n-1}$.  Again, it is not hard to show 
that, in the above notation the overlap of $a_2$ with
$s$ is maximal with respect to $\rho$. 
For $n\ge 2$, we associate to a path $a_n\in \AP(n)$, its {\it admissible
sequence (with respect to $\rho$}), which is defined to be $(p_{n-1},p_{n-2},\dots, p_1)$, where 
\begin{enumerate}
\item Each $p_i\in\rho=\AP(2)$,
\item $a_n=p_{n-1}s_{n-1}a_{n-2}$, for some $s\in\cB$ and $a_{n-2}\in\AP(n-2)$,
\item $a_n=ra_{n-1}$, for some $r\in \cB$ and $a_{n-1}\in \AP(n-1)$, and, if $n\ge 3$, then
\item $(p_{n-2},\dots,p_1)$ is the admissible sequence for $a_{n-1}$ with respect to $\rho$.
\end{enumerate}

\begin{prop}\label{apn}{\rm \cite{gs,ghz}} Let $\Lambda=K\Gamma/I$, for some
homogeneous ideal of $K\Gamma$.  Fix some admissible order $>$ on $\cB$, and
let $\cG$ denote the reduced \grb basis for $I$ with respect to $>$.  For
$n\ge 0$, let $\{f^n_i\}_{i\in T_n}$ denote the elements defined in \cite{gs}
in the construction of the AGS $\Lambda$-resolution
of $\Lambda_0$. Let $\AP(n)$ be the set of admissible paths of order $n$ with respect
to $\tip(\cG)$.  Then, for $n\ge 0$, 
\[\{\tip(f^n_i)\}_{i\in T_n}=AP(n).\]
\end{prop}

Finally, the following result relates the AGS $\Lambda$-resolution of $\Lambda_0$ to
a minimal projective $\Lambda_{mon}$-resolutions of $(\Lambda_{mon})_0$.
The proof follows from \cite{ghz}.

\begin{prop}\label{gr-hp-z} Let $\Lambda=K\Gamma/I$, where 
$I$ is a homogeneous ideal in $K\Gamma$ generated by elements of homogeneous length
at least $2$. Fix an admissible order $>$ on $\cB$ and let $\AP(n)$ denote the
admissible paths of order $n$ with respect to $\tip(\cG)$, where $\cG$ is
the reduced \grb basis of $I$ with respect to $>$.  Let $
\cdots \to L^2\to L^1\to L^0\to (\Lambda_{mon})_0\to 0$ be a minimal projective 
$\Lambda_{mon}$-resolution
of $(\Lambda_{mon})_0$ and $\cdots \to Q^2\to Q^1\to Q^0\to\Lambda_0\to 0$ be the
AGS $\Lambda$-resolution of $\Lambda_0$.  Then 
\[L^n\cong \bigoplus_{p\in \AP(n)}\Lambda_{mon} o(p).\]
and
\[Q^n\cong \bigoplus_{p\in \AP(n)}\Lambda o(p).\]
\end{prop}

\section{$d$-Koszul algebras}\label{d koz}
Fix the following notation for the remainder of this section.  We let $K$ denote a field,
$\Gamma$ a quiver, $K\Gamma$ the path algebra, $I$ a homogeneous ideal in $K\Gamma$ contained
in $J^2=\langle \text{arrows in }\Gamma\rangle^2$,
$\Lambda=K\Gamma/I$, which is given the induced length grading, $>$ an admissible ideal,
and $\cG=\{g^2_i\}_{i\in\cI}$ is the reduced \grb basis for $I$ with respect to
$>$, where $\cI$ is an index set.
Let $\cdots \to P^{2}_{\Lambda}\to P^1_{\Lambda}\to 
P^0_{\Lambda}\to \Lambda_0\to 0$ be a minimal graded
projective $\Lambda$-resolution. 
Recall that if $F\colon \mathbb N\to \mathbb N$ and, 
for each $n\ge 0$, $P^n$ can be
generated in degree $F(n)$, we say that $\Lambda$ is {\it $F$-determined}. If, for each
$n\ge 0$, $P^n$ can be generated in degees $\le F(n)$, we say that $\Lambda$ is {\it weakly
$F$-determined}.  If
$\delta\colon\mathbb N\to \mathbb N$ is
defined by
 \[\delta(n)=\lbrace \begin{array}{l}
 \frac{n}{2}d\mbox{ if }n\mbox{ is even}\\ \\ \frac{n-1}{2}d +1\mbox{ if }n\mbox{ is odd}
 \end{array}, \]
and $\Lambda$ is $\delta$-determined, we say that $\Lambda$ is {\it $d$-Koszul}. Note that
we use the term `$d$-Koszul' in this case, since the Ext-algebra, $\bigoplus_{n\ge 0}
\Ext_{\Lambda}^n(\Lambda_0,\Lambda_0)$, is finitely generated (in degrees $0$,$1$, and $2$)
\cite{gmmz}.  Also note that, since $\delta(2)=d$, the ideal $I$ is homogeneous and can
be generated by homogeneous elements of degree $d$.

In \cite{ag}, there are a number of results that have the form, if $\Lambda_{mon}$ has
some property, then so does $\Lambda$.  The next result is of this nature, and is useful
in both this section and the next, where we study a special class weakly $\delta$-determined algebras. 

\begin{prop}\label{f-det}
Let $\Lambda=K\Gamma/I$ be as above.
Suppose that 
$F\colon \mathbb N\to \mathbb N$ is a set function such the $\Lambda_{mon}$
is weakly $F$-determined.  Then $\Lambda$ is weakly $F$-determined. Furthermore, 
if $F$ is
not strictly increasing, then $\Lambda_0$ has finite projective dimension both as
a $\Lambda$-module, and as a $\Lambda_{mon}$-module. In particular, if $F(s+1)>F(s)$, for
$0\le s\le m$, and $F(m+1)\le F(m)$, then the projective dimensions
of $\Lambda_0\le m$. 

Moreover, if $\Lambda_{mon}$
is $F$-determined, then $\Lambda$ is $F$-determined and $\Lambda_{mon}$
is weakly $F$-determined and the AGS $\Lambda$-resolution of $\Lambda_0$ is minimal.
 \end{prop}
 
\begin{proof} First assume that $\Lambda_{mon}$ is either weakly $F$-determined or
$F$-determined.
Let $\cdots \to L^2\to L^1\to L^0\to (\Lambda_{mon})_0\to 0$ be the graded AGS 
$\Lambda_{mon}$-resolution of $(\Lambda_{mon})_0$, which is a minimal projective
$\Lambda_{mon}$-resolution of $(\Lambda_{mon})_0$.  
Let $\cdots \to P^2\to P^1\to P^0\to \Lambda_0\to 0$ be a minimal graded 
$\Lambda$-resolution of $\Lambda_0$ and $\cdots \to Q^2\to Q^1\to Q^0\to \Lambda_0\to 0$ 
be the graded AGS $\Lambda$-resolution of $\Lambda_0$.  Let $\{f_i^n\}_{i\in T_n}$ be
the elements constructed in the AGS $\Lambda$-resolution of $\Lambda_0$.  Then,
by Proposition \ref{genl-overlp}, for each $n\ge 0$, there are subsets $T^*_n$
of $T_n$ such that $P^n\cong \oplus_{j\in T^*_n}\Lambda o(f^n_i)[-\ell(f^n_i)]$.
By Proposition \ref{gr-hp-z}, for each $n\ge 0$, $L^n\cong \oplus_{i\in T_n}
\Lambda_{mon}o(\tip(f^n_i))[-\ell(\tip(f^n_i))]$.  Since $\Lambda_{mon}$ is either
weakly $F$-determined
or $F$-determined, we see that, for each $n\ge 0$, either $\ell(\tip(f^n_i))\le F(n)$ or
$\ell(\tip(f^n_i))= F(n)$, for
all $i\in T_n$.  Since $\ell(f^n_i)=\ell(\tip(f^n_i))$, which is the degree of $f^n_i$, 
for every $i$ and $n$, 
we see that $\Lambda$ is weakly $F$-determined if $\Lambda_{mon}$ or $F$-determined
if $\Lambda_{mon}$ is.

Now assume further that $F(s+1)>F(s)$, for
$0\le s\le m$ and $F(m+1)\le F(m)$.  Since 
 $\cdots \to L^2\to L^1\to L^0\to (\Lambda_{mon})_0\to 0$ is minimal graded AGS 
$\Lambda_{mon}$-resolution of $(\Lambda_{mon})_0$, we know that the image
of $L^n$ in $L^{n-1}$ is contained in $J/I_{mon}L^{n-1}$, for $n\ge 1$.
In particular, if $L^n\ne (0)$, then $F(n)\ge F(n-1)+1$.  Thus, $L^{m+1}=(0)$.
Hence, $T_{m+1}=\emptyset$.  Thus, we also have $P^{m+1}=Q^{m+1}=(0)$ and we
see that the projective dimension of $\Lambda_0$, as a $\Lambda$-module, and the
projective dimension of $(\Lambda_{mon})_0$, as a $\Lambda_{mon}$-module, 
is less than or equal to $m$.

It remains to show that, assuming that $\Lambda_{mon}$ is $F$-determined,
then the AGS $\Lambda$-resolution of $\Lambda_0$ is, in fact,  minimal.
By the argument given above, the minimality of the AGS $\Lambda_{mon}$-resolution of
$(\Lambda_{mon})_0$, implies that, if $Q^n\ne (0)$,
$F(n)\ge F(n-1)+1$.  But then the image of $Q^n$ in $Q^{n-1}$ is contained
in $(J/I) Q^{n-1}$ and the result follows.
\end{proof}

We remark that if $\cdots  \to P^2\to P^1\to P^0\to \Lambda_0\to 0$ be a minimal graded 
$\Lambda$-resolution of $\Lambda_0$, then an easy induction argument shows that if $P^n\ne 0$, then the
generators of $P^n$ occur in degrees $n$ or higher.  
Thus, we always assume that the functions $F$ under
consideration have the property that $F(n)\ge n$, for all $n\ge 0$.  Moreover, we assume, without loss of
generality, that each $P^i$ can be generated degrees greater than or equal to $i$, since the only possible
exceptions occur if $P^i=0$.  These remarks and conventions, allow us to redefine weakly $F$-determined to
mean that, for each $n\ge 0$, $P^n$ can be generated in degrees bounded below by $n$, and bounded above by $F(n)$.

 We now apply Proposition \ref{f-det} to the $d$-Koszul case.

\begin{cor}\label{d-kosz-cor} Let $\Lambda = K\Gamma/I$, where $I$ is a 
homogeneous ideal in $K\Gamma$ contained in $J^2$. 
If, for some admissible order, the
associated monomial algebra $\Lambda_{mon}$ is $d$-Koszul, then $\Lambda$ is $d$-Koszul
algebra, and the reduced Gr\"obner basis of $I$ is concentrated in degree $d$. Moreover, 
the AGS $\Lambda$-resolution of $\Lambda_0$ is a minimal graded projective
$\Lambda$-resolution. 
\end{cor}

Surprisingly, the following partial converse of Proposition \ref{f-det} is true.

\begin{prop}\label{conv-f-det}
Let $\Lambda=K\Gamma/I$ be as above.
Suppose that 
$F\colon \mathbb N\to \mathbb N$ is a set function such that $\Lambda$
is $F$-determined and assume that the AGS $\Lambda$-resolution of $\Lambda_0$ is 
minimal.  Then $\Lambda_{mon}$ is $F$-determined.
\end{prop}

\begin{proof} The proof follows from Proposition \ref{apn} and
Proposition \ref{gr-hp-z}.
\end{proof}

In general, homological properties of $\Lambda$ do not translate to
$\Lambda_{mon}$, but the above result and the next result are exceptions to this.
We now state the converse to Corollary \ref{d-kosz-cor} and note that we
do not assume that the AGS $\Lambda$-resolution of $\Lambda_0$ is minimal.  

\begin{thm}\label{monomial} Let $\Lambda = K\Gamma/I$ where $I$ is a homogeneous
ideal in $K\Gamma$.  Assume that $>$ is an admissible order on $\cB$ such that
the reduced \grb basis of $I$ is concentrated in degree $d$, where $d$ is a positive
integer greater than $1$.  Then 
 $\Lambda$ being $d$-Koszul implies that $\Lambda_{mon}$ is $d$-Koszul. 
 \end{thm} 

\begin{proof}
Let $\cG$ be the reduced \grb basis of $I$ with respect to $>$, $T_2$ be the
set of tips of $\cG$, and 
$T_3$ be the set of maximal overlaps of $T_2$ with respect to $\cG$.  
Since the paths in $T_2$ are
a set minimal generators of $I_{mon}$ and every path in $T_2$ is of length $d$,
by \cite{gmmz}, we need only show that every element of $T_3$ is of length
$d+1$.

Let $g_1,g_2\in \cG$ and $t_i=\tip(g_i)$, for $i=1,2$.  Assume that $t_1$ maximally
overlaps $t_2$ with respect to $\cG$ with $p=t_2s=rt_1$ the maximal overlap, where $r,s\in \cB$.
Assume that $\ell(p)>d+1$ and let $\ell(p)=d^*$. Note that $d^*<2d$. 
We show that this assumption
leads to a contradiction.  Clearly,
$rg_1-g_2s$ is a homogeneous element of degree $d^*$.  Since $\cG$ is a \grb basis
of $I$, there exist elements nonnegative integers, $A$ and $B$, and 
elements, $\alpha_1,\dots, \alpha_A,\beta_1,\dots,\beta_B\in K$, 
$x_1,\dots, x_A, y_1,\dots, y_B, z_1,\dots, z_B\in \cB$, with $\ell(z_j)\ge 1$,\
for $j=1,\dots,B$,
and $g^{'}_1,\dots, g^{'}_A,
g^{''}_1,\dots, g^{''}_B\in \cG$,
such that
\[(*)\quad \quad rg_1-g_2s=\sum_{i=1}^A\alpha_ix_ig^{'}_i +\sum_{j=1}^B\beta_jy_jg^{''}_jz_j,\]
where $\tip(rg_1-g_2s)<r\tip(g_1)$. Since $\cG$ is a reduced \grb basis, 
the right hand side of (*) is unique.

Since the elements of $\cG$ are all homogeneous of degree $d$, we may assume
that $\ell(x_i)=d^*-d$, for $i=1,\dots, A$ and that
$\ell(y_j)+\ell(z_j)=d^*-d$, for $j=1,\dots,B$. 
Now let 
\[F =r g_1 - \sum_{i=1}^A\alpha_i  x_ig^{'}_i =  g_2s + \sum_{j=1}^B 
\beta_jy_j g^{''}_j z_j.\] 
We see that  $\tip(F)= rt_1$,  which is a path of length $d^*$.

Let $\{f^n_i\}_{i\in T_n}$ be the given by the AGS $\Lambda$-resolution
as in \cite{gs}.  As shown in \cite{gsz}, we may obtain sets $V_n$
with $V_n\subset T_n$ so that $\{f^n_i\}_{i\in V_n}$
correspond to a minimal graded projective $\Lambda$-resolution of $\Lambda_0$.
As we stated in Section \ref{ags sec}, we have, for each $n\ge 2$, 
\[(**)\quad\quad (\oplus_{i\in V_{n-1}}K\Gamma f^{n-1}_i)
\cap (\oplus_{i\in V_{n-2}}I{f^{n-2}_i}) =(\oplus_{i\in V_n}K\Gamma f^n_i)
\oplus(\oplus_{i\in W_n}R{f^n_i}'),\]
where each ${f^n_i}'$ is in $\oplus_{i\in W_{n-1}}I f^{n-1}_i$, and the $W_n$ 
are given in \cite{gs} and \cite{gsz}.  Now, for $i=0,1$, $V_i=T_i$, since
$\Gamma_i$ is the set of $f^i_j$'s
in both the minimal and the
AGS $\Lambda$-resolutions for $\Lambda_0$.  Our assumption that the reduced
\grb basis of $I$ consists of homogeneous elements in one degree $d$, implies
that $V_2=T_2$, since, in this case, $\cG$ is the set of $f^2_j$'s for $\Lambda_0$
in both the minimal and the
AGS $\Lambda$-resolutions for $\Lambda_0$.
Applying (**), we see that
\[(\oplus_{g\in \cG}K\Gamma g)
\cap (\oplus_{a\in \Gamma_1}Ia) =(\oplus_{i\in V_3}K\Gamma f^3_i)
\oplus(\oplus_{i\in W_3}R{f^3_i}'),\]
From the definition of $F$, we see that $F\in (\oplus_{g\in \cG}K\Gamma g)
\cap (\oplus_{a\in \Gamma_1}Ia)$.  Thus,
\[F=\sum_{i\in V_3}h_i f^3_i +\sum_{i\in W_3}h^{'}_i{f^3_i}',\]
for some $h_i,h^{'}_j\in K\Gamma$.
But, since each ${f^3_i}'\in \oplus_{g\in\cG}I g$ and $I$ is generated by
$\cG$, we conclude that the each ${f^3_i}'$ is homogeneous of length at least $2d$.
Noting that $2d>d^*$, we conclude that each $h^{'}_j=0$ and
\[F=\sum_{i\in V_3}h_i f^3_i.\]
Our assumption that $\Lambda$ is $d$-Koszul, implies 
that each $f^3$ is homogeneous of length $d+1$.
Thus each $h_i$ is homogeneous of length $d^*-d-1\ge 1$.
But then $\tip(F)=\tip(h_i)\tip(f^3_i)$, for some $i\in V_3$.
Now $\tip(F)$ is the maximal overlap
of tips of $\cG$ of length $d^*$ and 
$\tip(f^3_i)$ is a maximal overlap of tips of $\cG$ of
length $d+1$. This is a contradiction since distinct maximal overlaps
of elements of $\tip(\cG)$ cannot be subwords of one another.
This completes the proof. 
\end{proof}

A consequence of the results of this section is that, if the reduced \grb
basis of an ideal $I$  consists of elements, all homogeneous of one degree, then
there is a finite check to determine whether or not $K\Gamma/I$ is $d$-Koszul.

\begin{prop}\label{d-max}
Let $\Lambda = K\Gamma/I$ and $>$ an admissible order such that the reduced Gr\"obner
basis of $I$ is concentrated in degree $d$. Then  $\Lambda$ is $d$-Koszul if and only if 
the set of maximal overlaps of elements of $\tip(\cG)$ with respect
to $\cG$ are all of length $d+1$.
\end{prop}

\begin{proof} Let $\cG$ be the reduced \grb basis of $I$ with respect
to $>$. If $\Lambda$ is $d$-Koszul, the proof of the above theorem shows
that every maximal overlap of tips of $\cG$ is a path of length $d+1$.
On the other hand, if the set of maximal overlaps of elements of 
$\tip(\cG)$ are all of length $d+1$, then since $I_{mon}$ has $\tip(\cG)$ as
its minimal generating set, by  \cite[Theorem 10.2]{gmmz}, $\Lambda_{mon}$
is $d$-Koszul.  Then $\Lambda$ is $d$-Koszul by Corollary \ref{d-kosz-cor}.
\end{proof}

The following result summarizes the main ideas of this section.

\begin{thm}\label{summ}Let $\Lambda = K\Gamma/I$ and $>$ an admissible order such that the reduced Gr\"obner
basis of $I$ is concentrated in degree $d$, with $d\ge 2$. Then the following statements are equivalent:
\begin{enumerate}
\item $\Lambda$ is a $d$-Koszul algebra.
\item $\Lambda_{mon}$ is a $d$-Koszul algebra.
\item If $\cdots \to P^2\to P^1\to P^0\to \Lambda_0\to 0$ is a minimal graded $\Lambda$-projective
resolution of $\Lambda_0$, then $P^3$ is generated in degree $d+1$.
\item If $\cdots \to P^2\to P^1\to P^0\to (\Lambda_{mon})_0\to 0$ is a minimal graded 
$\Lambda_{mon}$-projective
resolution of $(\Lambda_{mon})_0$, then $P^3$ is generated in degree $d+1$.
\item If $\cG$ is the reduced \grb basis of $I$ with respect to $>$, the 
every maximal overlaps of two elements of $\{\tip(\cG)\}$ with respect
to $\cG$, is of length $d+1$.

\end{enumerate}

\end{thm}

\section{ $2$-$d$-determined monomial algebras are $2$-$d$-Koszul}\label{2-d det}
Let $\Lambda=K\Gamma/I$, where $I$ is a homogeneous ideal.
We keep the convention that $\delta\colon\mathbb N\to \mathbb N$ is
defined by
 \[\delta(n)=\lbrace \begin{array}{l}
 \frac{n}{2}d\mbox{ if }n\mbox{ is even}\\ \\ \frac{n-1}{2}d +1\mbox{ if }n\mbox{ is odd}
 \end{array}. \]
We also let $\cdots \to P^2\to P^1\to P^0\to \Lambda_0\to 0$ be
a minimal graded
projective $\Lambda$-resolution of $\Lambda_0$.  We say that $\Lambda$
is {\it $2$-$d$-determined} if $\Lambda$ is weakly $\delta$-determined; that is, 
for each $n\ge 0$,  $P^n$ can be generated by
elements of degree at least $n$ and not greater than $\delta(n)$.  

In keeping with
the philosophy that the use of the word `Koszul' should imply that Ext-algebra is
finitely generated, we say that a $2$-$d$-determined algebra is a {\it $2$-$d$-Koszul
algebra} if its Ext-algebra, $\bigoplus\Ext_{\Lambda}^{n\ge 0}(\Lambda_0,\Lambda_0)$ is
finitely generated.  We prove later in this section that if $\Lambda$ is 
$2$-$d$-determined monomial algebra, then $\Lambda$ is a $2$-$d$-Koszul algebra;
in particular, we show that the Ext-algebra of $\Lambda$ can be generated in degrees
$0$, $1$, and $2$.

For the remainder of this section, we restrict 
our attention to monomial algebras such that the minimal generating
set of monomial relations occur in exactly two degrees, $2$ and $d$, where $d$ is
an integer greater than $2$. We fix the following notation for the remainder
of this section. Let $d$ be an integer greater than $2$, 
$K$ is a field, $\Gamma$ is a quiver, 
$I$ is a monomial ideal generated
by paths of length $2$ and paths of length $d$ and $\Lambda=K\Gamma/I$.
Since $I$ is generated by monomials, there is a unique minimal set of generating
paths, $\cG$, such that $\cG$ is the reduced \grb basis of $I$ with respect to
any admissible order on $\cB$.
Let $\cG_2$ denote the set
of paths of length $2$ in $I$.  Let $\cG_d$ denote the set of paths of
length $d$ in $\cG$. Note that $\cG$, and hence $\cG_2$ and $\cG_d$ are 
independent of the choice of $>$ and that our assumption
that $I$ is a monomial ideal generated in degrees $2$ and $d$ implies
that $\cG_2\cup \cG_d = \cG$.

Our first result gives necessary and 
sufficient conditions for the monomial algebra 
$\Lambda$ to be $2$-$d$-determined.  Before giving the result, we recall Theorem 10.2
from \cite{gmmz}.

\begin{prop}\label{overlap-d-koz} Let $\Lambda=K\Gamma/I$, where $I$ is
a monomial ideal generated by a set, $\rho$, of paths of length $d$ with $d\ge 2$.  
Then $\Lambda$
is $d$-Koszul algebra if and only if, for each pair of paths
$p,q\in \rho$, if $pr=sq$ with $1\le\ell(r)<d$ then every subpath of $pr$ of length
$d$ is in $\rho$.
\end{prop}

 \begin{thm}\label{mono-2d}Keeping the notations above, 
 $\Lambda$ is $2$-$d$-determined if and only if the algebra $\Delta$= $K\Gamma/<\cG_d> $ 
 is a $d$-Koszul algebra. 
 \end{thm} 
 
 \begin{proof}  For $n\ge 0$, let $\AP_{\Delta}(n)$ be the admissible sets for $\rho=\cG_d$,
 defined in Section \ref{ags sec} and $\AP_{\Lambda}(n)$ be the admissible
 sets for $\rho=\cG$.  
 
 Let $n\ge 0$ and assume that $\Delta$ is a $d$-Koszul monomial algebra.
We need to show that if $a_n\in\AP_{\Lambda}(n)$,
 then $n\le \ell(a_n)\le \delta(n)$.  By definition of the $\AP(n)$'s, the inequalities
 hold for $n=0,1,2$.  Assume by induction, that $n\ge 3$ and the inequalities hold for
 $n-2$ and $n-1$.
There 
are unique elements 
\[a_{n-1}\in \AP(n-1), a_{n-2}\in\AP(n-2), r,s\in \cB,\] such
that $a_n=ra_{n-1}$ and $a_{n-1}=sa_{n-2}$.  Furthermore, there is some $a_2\in\AP(2)$
such that $a_2$ maximally overlaps $rs$ with respect to $\cG$ and $a_n=a_2ta_{n-2}$, for some path
$t$.  If $\ell(a_2)=2$, then $\ell(a_n)=\ell(a_{n-1}) +1$ and the result follows
from induction.  If $\ell(a_2)=d$ and $\ell(s)=1$, then again the result follows.
Finally if $\ell(a_2)=d$ and $\ell(s)>1$, then $\ell(s)\le d-1$ and it follows
that $s$ is a prefix of an element $a_2'$ of $\AP(2)$ of length $d$. 
Hence $a_2$ overlaps $a'_2$.  By Proposition \ref{overlap-d-koz} and the maximality
of the overlap with respect to $\cG$, we have that
$a_2=rs$.  Thus $\ell(a_n)=d+\ell(a_{n-2})$. Our assumption
implies that 
$\ell(a_{n-2})\le \delta(n-2)=((n-2)d/2) +1$ if $n$ is even and
$\ell(a_{n-2})\le \delta(n-2)=((n-3)d/2)$ if $n$ is odd, and the result follows.

Now we assume that $\Lambda$ is $2$-$d$-determined.  By Proposition \ref{overlap-d-koz}, it
suffices to show that if $a_2, a'_2\in \cG_d$ and $a_2$ overlaps $a'_2$, then
every subpath of length $d$ is in $\cG_d$.  Suppose that $a_2r=sa'_2$ with $\ell(r)\ge 1$.
We proceed by induction on the length of $r$.
If $\ell(r)=1$, then we are done.  Suppose that $\ell(r)>1$ and that 
$a_2r=\alpha_{\ell(a_2r)}\alpha_{\ell(a_2r)-1}\cdots\alpha_{2}\alpha_1$, with the 
$\alpha_j$'s arrows. It suffices to show that  $a^*_2=\alpha_{d+1}
\alpha_d\cdots \alpha_{2}\in \cG_d$, since,
if so, $a_2$ overlaps $a^*_2$ with $a_2r^*=s^*a^*_2$ and $\ell(r^*)=\ell(r)-1$.
Since $\ell(r)>1$,  $\ell(a_2r)>\delta(3)=d+1$.  Hence $a_2r\notin \AP(3)$ and we conclude
that the overlap of $a_2$ with $a'_2$ is not maximal with respect to $\cG$.  Thus there is
$\hat{a}_2\in \cG$ that maximally overlaps $a'_2$ with respect to $\cG$.  If $\ell(\hat{a}_2)=2$,
then $\hat{a}_2$ is a subpath of $a_2$, contradicting that $\cG$ is a reduced \grb basis.
Thus, $\hat{a}_2$ has length $d$ and, since $\delta(3)=d+1$, we see that $\hat{a}_2=\alpha_{d+1}\cdots \alpha_2$,
as desired.
\end{proof}

Since $\cG_d$ being the \grb basis of $d$-Koszul algebra is equivalent to
the length of every element in $\AP(3)$ (with respect to $\cG_d$) having length
exactly $d+1$, we have the following consequence of Theorem \ref{mono-2d}.

\begin{cor}\label{mono-cor} Let
$\Lambda = K\Gamma/I$ be as in Theorem \ref{mono-2d}. The following statements are
equivalent:

\begin{enumerate} 
\item The algebra $\Lambda$ is $2$-$d$-determined.
\item $P_3$ can be generated in
degrees bounded above by $d+1$. 
\end{enumerate} \end{cor}

\begin{proof}
Let $a_2,a'_2\in \cG=\AP(3)$, where $\AP(n)$ are the admissible paths of order $n$ with
respect to $\cG$.  Suppose that $a_2$ maximally overlaps $s'_2$ with 
respect to $\cG$ and let $p\in \AP(3)$ be the
overlap.  If either $\ell(a_2)=2$ or $\ell(a'_2)=2$, then $\ell(p)\le d+1=\delta(3)$.
If both $a_2$ and $a'_2$ are of length $d$, then the previous theorem and the properties
of $d$-Koszul monomial algebras.
\end{proof}

We now turn our attention to the Ext-algebra of a monomial $2$-$d$-determined algebra.
Recall that the {\it Ext-algebra of $\Lambda$}, which we denote by $E(\Lambda)$, is the algebra
$\bigoplus_{n\ge 0}\Ext^n_{\Lambda}(\Lambda_0,\Lambda_0)$.
We view $E(\Lambda)$ as a positively $\mathbb Z$-graded algebra, where
$E(\Lambda)_n=\Ext^n_{\Lambda}(\Lambda_0,\Lambda_0)$.
The next result  shows that every $2$-$d$-determined monomial algebra is a $2$-$d$-Koszul
algebra.

\begin{thm}\label{2d-mono-koz} Let $\Lambda=K\Gamma/I$, where $I$ is a monomial ideal
generated by paths of $2$ and $d$ with $d\ge 3$.  Then $\Lambda$ is $2$-$d$-determined
if and only if $\Lambda$ is $2$-$d$-Koszul.  Moreover, $E(\Lambda)$ can be generated 
in degrees $0$, $1$, and $2$.
\end{thm}

\begin{proof}
It suffices to show that if $\Lambda$ is a $2$-$d$-determined monomial algebra, then
$E(\Lambda)$ can be generated 
in degrees $0$, $1$, and $2$. Let $\cG$ be the reduced \grb basis of $I$ with respect
to some admissible order.   We let $\cG=\cG_2\cup \cG_d$, where $\cG_2$ are the elements
of $\cG$ of degree $2$ and $\cG_d$ are the elements of $\cG$ of degree $d$.
Let $\AP(n)$ be the admissible paths with respect to $\cG$. 
Suppose $n\ge 3$ and $a_n\in\AP(n)$.  There 
are unique elements 
\[a_{n-1}\in \AP(n-1), a_{n-2}\in\AP(n-2), r,s\in \cB,\] such
that $a_n=ra_{n-1}$ and $a_{n-1}=sa_{n-2}$.  Furthermore, there is some $a_2\in\AP(2)$
such that $a_2$ maximally overlaps $rs$ with respect to $\cG$ and $a_n=a_2ta_{n-2}$, for some path
$t$. We also have that there is some $a_2^*\in \AP(2)$ such that $a_{n-1}=a_2^*p$, for some
path $p$.

By \cite{gz}, it suffices to show that either $a_n=a_2a_{n-2}$, 
or $a_n=a_1a_{n-1}$, for some $a_1\in \AP(1)=\Gamma_1$.
If $\ell(a_n)=\ell(a_{n-1})+1$, then $a_n=a_1a_{n-1}$, for some $a_1\in \AP(1)=\Gamma_1$
and we are done.  Suppose that $\ell(a_n)\ge \ell(a_{n-1})+2$.  Then $a_2$ has length
$>2$ and we see that $\ell(a_2)=d$.

We must show
that $\ell(t)=0$; that is, we must show that $t$ is a vertex.
If $\ell(t)>0$, then $\ell(a_{n-1})>\ell(a_{n-2})+1$ and we conclude that
$\ell(a_2^*)=d$.  Hence, suppose that $\ell(a_2^*)=d$. 
We know that $a_2$ maximally overlaps $rs$ with respect to $\cG$.
We have that $a_2^*=ss'$ for some path $s'$.  Thus $a_2$ overlaps $a_2^*$.  By our
assumption that $\Lambda$ is $2$-$d$-determined, by Theorem \ref{mono-2d}, we
see that $K\Gamma/\langle\cG_d\rangle$ is $d$-Koszul.  Hence, by Proposition
\ref{overlap-d-koz}, every path of length $d$ in the overlap of $a_2$ with $a_2^*$
is in $\cG_d\subset \cG$.  In particular, this implies that $t$ must be a vertex.
This completes the proof.
\end{proof}

\section{$2$-$d$-determined algebras and \grb bases}\label{2d-grobner}

Let $d>2$, $I$ be a homogeneous ideal in $K\Gamma$, and $\Lambda=K\Gamma$.  In this section
we study conditions on $\Lambda$ that imply that 
$\Lambda$ is $2$-$d$-determined.  At this time, we do not know
if every $2$-$d$-determined algebra is a $2$-$d$-Koszul algebra.  Throughout this
section, we 
fix an admissible order $>$ on $\cB$.  Since, by definition, a $2$-$d$-determined algebra
is just a weakly $\delta$-determined algebra, the next result is an immediate consequence
of Proposition \ref{f-det}.

\begin{prop}\label{gen-2d} Keeping the above notations, if $\Lambda_{mon}$ is
$2$-$d$-Koszul then $\Lambda$ is $2$-$d$-determined.
\end{prop}

We now prove another result that gives sufficient conditions
for $\Lambda$ to be $2$-$d$-determined.

\begin{thm}\label{striking} 
Let $\Lambda = K\Gamma/I$, where $I$ is a homogeneous 
ideal in $K\Gamma$, and let $>$ be an admissible order on $\cB$.  Suppose that the reduced
Gr\"obner basis $\cG$  of $I$ with respect to $>$ satisfies $\cG =\cG_2 \cup \cG_d$
where $\cG_2$ consists of homogeneous elements of degree $2$ and
$\cG_d$ consists of homogeneous elements of degree $d$, where $d\ge 3$.
Then $\Lambda$ is $2$-$d$-determined
if $K\Gamma/\langle\tip( \cG_d)\rangle$ is a $d$-Koszul algebra.   
\end{thm}

\begin{proof} We begin by showing that if $K\Gamma/\langle\tip(\cG_d)\rangle$ is
a $d$-Koszul algebra, then $\Lambda_{mon}$ is $2$-$d$-determined. By Theorem
\ref{mono-2d} and Proposition \ref{gen-2d},
this then shows that if $K\Gamma/\langle\tip(\cG_d)\rangle$ is
a $d$-Koszul algebra, then $\Lambda$ is $2$-$d$-determined.
Assume  $K\Gamma/\langle\tip(\cG_d)\rangle$ is
a $d$-Koszul algebra and,
for $n\ge 0$, let $\AP_{\Delta}(n)$ be the admissible sets for $\rho=\tip(\cG_d)$,
 defined in Section \ref{ags sec}. Let $\AP_{\Lambda_{mon}}(n)$ be the admissible
sets for $\rho=\tip(\cG)$.

We need to show that if $a_n\in\AP_{\Lambda_{mon}}(n)$,
 then $n\le \ell(a_n)\le \delta(n)$.  By definition of the $\AP_{\Lambda_{mon}}(n)$'s, 
the inequalities
hold for $n=0,1,2$.  Assume by induction, that $n\ge 3$ and the inequalities hold for
$n-2$ and $n-1$.
There 
are unique elements 
\[a_{n-1}\in \AP_{\Lambda_{mon}}(n-1), a_{n-2}\in\AP_{\Lambda_{mon}}(n-2), r,s\in \cB,\] 
such that $a_n=ra_{n-1}$ and $a_{n-1}=sa_{n-2}$.  
Furthermore, there is some $a_2\in\AP_{\Lambda_{mon}}(2)$
such that $a_2$ maximally overlaps $rs$ with respect to $\tip(\cG)$ 
and $a_n=a_2ta_{n-2}$, for some path
$t$.  If $\ell(a_2)=2$, then $\ell(a_n)=\ell(a_{n-1}) +1$ and the result follows
from induction.  If $\ell(a_2)=d$ and $\ell(s)=1$, then again the result follows.
Finally, if $\ell(a_2)=d$ and $\ell(s)>1$, then $\ell(s)\le d-1$ and it follows
that $s$ is a prefix of an element $a_2'$ of $\AP_{\Lambda_{mon}}(2)$ of length $d$. 
Hence $a_2$ overlaps $a'_2$.  By Proposition \ref{overlap-d-koz} (applied to 
the $d$-Koszul algebra 
$K\Gamma/\langle\tip(\cG_d)\rangle$)
and by the maximality
of the overlap with respect to $\tip(\cG)$, we have that
$a_2=rs$.  Thus $\ell(a_n)=d+\ell(a_{n-2})$. Our assumption
implies that 
$\ell(a_{n-2})\le \delta(n-2)=((n-2)d/2) +1$ if $n$ is even, and
$\ell(a_{n-2})\le \delta(n-2)=((n-3)d/2)$ if $n$ is odd, and we have shown that
$\Lambda_{mon}$ is $2$-$d$-determined.
\end{proof}

 Suppose, as in the theorem
above, that $d\ge 3$ and $\cG$
is the reduced \grb basis for a homogeneous ideal and that $\cG=\cG_2\cup\cG_d$, where
$\cG_2$ consists of quadratic elements and $\cG_d$ consists of homogeneous elements of 
degree $d$.  If $d>3$, then  $\cG_2$ is the reduced \grb basis of the
ideal it generates and hence, by \cite{gh}, $K\Gamma/\langle\cG_2\rangle$ is
a Koszul algebra.  On the other hand, if $d=3$, then $\cG_2$ need not be the reduced
\grb basis of the ideal it generates and it is not necessarily the
case that $K\Gamma/\langle\cG_2\rangle$ is
a Koszul algebra. 

The next result is a partial converse to the above theorem. 

\begin{prop}\label{striking-conv}Let $\Lambda = K\Gamma/I$, where $I$ is a homogeneous 
ideal in $K\Gamma$, and let $>$ be an admissible order on $\cB$.  Suppose that the reduced
Gr\"obner basis $\cG$  of $I$ with respect to $>$ satisfies $\cG =\cG_2 \cup \cG_d$
where $\cG_2$ consists of homogeneous elements of degree $2$ and
$\cG_d$ consists of homogeneous elements of degree $d$, where $d\ge 3$.
If $\Lambda$ is $2$-$d$-Koszul and the AGS $\Lambda$-resolution of $\Lambda_0$ is 
minimal, then $K\Gamma/\langle \tip(\cG_d)\rangle$ is a $d$-Koszul algebra.
\end{prop}

\begin{proof}
We follow a line of reasoning similar to 
that found in the proof of Theorem \ref{monomial}.
If $\{f^n_i\}_{i\in T_n}$ are given by the AGS $\Lambda$-resolution, which, we are assuming,
is a minimal graded projective $\Lambda$-resolution of $\Lambda_0$.
As we stated in Section \ref{ags sec}, we have, for each $n\ge 2$, 
\[(\oplus_{i\in T_{n-1}}K\Gamma f^{n-1}_i)
\cap (\oplus_{i\in T_{n-2}}I{f^{n-2}_i}) =(\oplus_{i\in T_n}K\Gamma f^n_i)
\oplus(\oplus_{i\in W_n}R{f^n_i}'),\]
where each ${f^n_i}'$ is in $\oplus_{i\in W_{n-1}}I f^{n-1}_i$, and the $W_n$ 
are given in \cite{gs} and \cite{gsz}.  We are assuming that the reduced
\grb basis of $I$ consists of homogeneous elements in two degrees, $2$ and $d$.
Our assumption that $\Lambda$ is $2$-$d$-determined, implies that every
$f^i_3\in V_3$ is homogeneous of degree $\le d+1$.  In fact, from the construction
of the $f_3^i$'s, the homogeneous degrees of elements in $T_3$ are either $3$ or
$d+1$.

Suppose that $\Lambda$ is $2$-$d$-determined and the AGS $\Lambda$-resolution
of $\Lambda_0$ is minimal. Let
$\Delta=K\Gamma/\langle\tip(\cG_d)\rangle$. We wish to show that $\Delta$ is
a $d$-Koszul algebra.  For this, let $\AP_{\Delta}(n)$ denote the admissible
sequences for $\tip(\cG_d)$.  It suffices to show that if $a_3\in \AP_{\Delta}(3)$,
then $\ell(a_3)=d+1$ by Proposition \ref{d-max}.  Let $a_3\in \AP_{\Delta}$ and
suppose that $a_2,a'_2\in\AP_{\Delta}(2)$
are such that $a_3$ is the maximal overlap of $a'_2$ with $a_2$.  Let $g_2,g'_2\in\cG_d$ such
that $\tip(g_2)=a_2$ and $\tip(g'_2)=a'_2$.
If $p$ and $q$ are paths such that $a_3=a'_2p=qa_2$, then, since
$\cG$ is a reduced (homogeneous) \grb basis, there exist homogeneous elements
$r_h,s_h,t_h\in K\Gamma$, for $h\in\cG$ such that
\[
qg_2-g'_2p=
\sum_{h\in\cG}r_hh+\sum_{h\in\cG}s_hht_h,
\]
with each term occurring has the same length and each nonzero $s_h$ has length $\ge 1$.
Let $X=qg_2-\sum_{h\in\cG}=g'_2p+\sum_{h\in\cG}s_hht_h$.  Then we see that $\tip(X)=a_3$, 
$X$ is one of the $f^3_i$'s and hence of length $\le d+1$.  It follows that 
$\ell(a_3)\le d+1$ and we are done.

\end{proof}

We can now put together our results in the following theorem.

\begin{thm}\label{final-summary}
Let $\Lambda = K\Gamma/I$, where $I$ is a homogeneous 
ideal in $K\Gamma$, and let $>$ be an admissible order on $\cB$.  Suppose that the reduced
Gr\"obner basis $\cG$  of $I$ with respect to $>$ satisfies $\cG =\cG_2 \cup \cG_d$
where $\cG_2$ consists of homogeneous elements of degree $2$ and
$\cG_d$ consists of homogeneous elements of degree $d$, where $d\ge 3$.
Then the following are true.
\begin{enumerate}
\item If $K\Gamma/\langle \cG_d\rangle$ is $d$-Koszul then $\Lambda$ is a $2$-$d$-determined
algebra and $K\Gamma/\langle \tip(\cG)\rangle$ is $2$-$d$-Koszul.
\item If $K\Gamma/\langle \tip(\cG)\rangle$ is $2$-$d$-Koszul, then $\Lambda$ is
$2$-$d$-determined.
\item If $\cdots \to P^3\to P^2\to P^1\to P^0\to (K\Gamma/\langle \cG_d\rangle)_0$ is
a minimal graded projective $K\Gamma/\langle \cG_d\rangle$-resolution and
$P^3$ can be generated in degree $\le d+1$, then $\Lambda$ is a $2$-$d$-determined algebra.

\end{enumerate}
Assuming the AGS $\Lambda$-resolution of $\Lambda_0$ is minimal, the following
statements are equivalent.
\begin{enumerate}
\item[(4)]  The algebra $\Lambda$ is $2$-$d$-determined.
\item[(5)]  The algebra $\Lambda_{mon}$ is $2$-$d$-Koszul.
\item[(6)]  The algebra $K\Gamma/\langle \tip(\cG_d)\rangle$ is $d$-Koszul.
\end{enumerate}
\end{thm}
\begin{proof}
By Theorem \ref{summ} ((1) implies (2)),  we see that $K\Gamma/\langle \cG_d\rangle$ being 
$d$-Koszul implies that $K\Gamma/\langle \cG_d\rangle_{mon}$ is $d$-Koszul.  But
$K\Gamma/\langle \cG_d\rangle_{mon}=K\Gamma/\langle \tip(\cG_d)\rangle$ and 
part (1) follows from Theorem \ref{striking}. Part (2) follows from Theorems \ref{mono-2d}
and \ref{striking}.  Part (3) follows from Corollary \ref{mono-cor} and Theorem \ref{striking}.

We have seen that (5) implies (6) and that (6) implies (4).  That (4) implies (5), follows
from Proposition \ref{striking-conv}.\end{proof}

We end with the obvious questions:

\noindent
{\bf Questions}: Assume that $\Lambda=K\Gamma/I$, where $I$ is a ideal generated by
homogeneous elements of degrees $2$ and $d$.
\begin{enumerate}
\item[(i)] If $\Lambda$ is a $2$-$d$-determined algebra, then, is the Ext-algebra 
$E(\Lambda)=$
\linebreak
$\oplus_{n\ge 0}\Ext^n_{\Lambda}(\Lambda_0,\Lambda_0)$ finitely generated?  
\item[(ii)] If $\Lambda$ is a $2$-$d$-determined algebra and the Ext-algebra
$E(\Lambda)$ finitely generated,
is it generated in degrees $0$, $1$, and $2$ (assuming that the global dimension of $\Lambda$
is infinite). 
\item[(iii)] If $\Lambda$ is not of finite global dimension and $E(\Lambda)$ is generated
in degrees $0$, $1$, and $2$, then is $\Lambda$ $2$-$d$-determined?
\end{enumerate}

\end{document}